\documentclass[11pt]{amsart}
\usepackage{amsaddr}
\usepackage{fullpage}
\usepackage{amsmath, amssymb,url}
\usepackage{amsthm}
\usepackage{textcomp}
\usepackage{graphicx}
\usepackage{latexsym, geometry}
\usepackage{color}
\usepackage[all]{xy}

\usepackage{enumitem}
\usepackage{mathtools}
\usepackage[latin1]{inputenc}
\usepackage{ytableau}
\usepackage[enableskew,vcentermath]{youngtab}
\usepackage{tikz}

\newcommand{\symm}[1]{\mathfrak{S}_{#1}}

\numberwithin{figure}{section}

\DeclareMathOperator{\depth}{depth}
\DeclareMathOperator{\arcs}{\text Arcs}
\DeclareMathOperator{\inumber}{\mathtt i}

\DeclareMathOperator{\dindex}{\mathtt t}
\DeclareMathOperator{\cro}{\mathtt cr}
\DeclareMathOperator{\nst}{\mathtt ne}
\DeclareMathOperator{\al}{\mathtt al}
\DeclareMathOperator{\aspan}{\mathtt span}
\DeclareMathOperator{\tvd}{\mathtt tvd}

\theoremstyle{plain}
\newtheorem{theorem}{Theorem}[section]

\newtheorem{lemma}[theorem]{Lemma}
\newtheorem{corollary}[theorem]{Corollary}

\newtheorem*{theorem*}{Theorem}
\newtheorem*{definition*}{Definition}
\newtheorem*{problem*}{Problem}
\newtheorem*{conjecture*}{Conjecture}
\theoremstyle{definition}
\newtheorem{definition}[theorem]{Definition}

\newtheorem{example}[theorem]{Example}

\newtheorem{remark}[theorem]{Remark}
\newtheorem{observation}[theorem]{Observation}

\begin{document}
\title{The distribution of the intertwining number on perfect matchings}

\author{Yonah Cherniavsky*}
\author{Yuval Khachatryan-Raziel}
	
\address[A1,A2]{Department of Mathematics, Ariel University, Israel}
\email[A1,A2]{yonahch@ariel.ac.il, yuvalkh@ariel.ac.il}
	
\normalsize

\date{\today} 
\maketitle
\begin{abstract}
Ehrenborg and Readdy defined the intertwining number on set partitions. Considering this statistic on perfect matchings, we provide an additional combinatorial description and give an explicit generating function. In particular, we show that the intertwining number on perfect matchings is essentially equidistributed with the rank function of the strong Bruhat order on fixed-point-free involutions of the symmetric group.

\end{abstract}

\section{Introduction}

This paper is concerned with the intertwining number of a set partition, which is a combinatorial statistic introduced by Ehrenborg and Readdy in~\cite{EhrenborgReaddy}. This statistic is among the combinatorial parameters on set partitions whose generating function is an important $q$-analog of the Stirling numbers of the second kind: 
\begin{align}\label{A:particular}
S_q(n,k)=
\begin{cases}
q^{k-1}S_q(n-1,k-1)+[k]_qS_q(n-1,k) & \ \text{ if $n,k\geq 1$};\\
\delta_{n,k} & \ \text{ if $n=0$ or $k=0$.}
\end{cases}
\end{align}
Here $\delta_{n,k}$ is the Kronecker's delta function.

This paper studies the distribution of the intertwining number on the set of perfect matchings denoted $PM_{2n}$. A perfect matching can be viewed as a set partition all of whose blocks have exactly two elements.
The intertwining number is closely related to the so called  depth-index. The  depth-index of a set partition was introduced and studied in \cite{StirlingPosets} by Can and the first author. In \cite{GeometricInterpretationIntertwiningNumber} together with Rubey they established a connection between the intertwining number and the depth-index of set partitions: the sum these two statistics equals to ${n\choose 2}$. 
As shown in~\cite{StirlingPosets}, the depth-index is equal to the rank function of a certain graded EL-shellable order on $\Pi_n$, the set of set partitions $[n]=\{1, 2,\ldots, n\}$. That order is isomorphic to the Bruhat-Chevalley-Renner order on upper-triangular matrices, and thus, the depth-index appears to be equal to the dimension of a certain matrix variety. Thus, by the result of \cite{GeometricInterpretationIntertwiningNumber}, the intertwining number also got a geometric interpretation.  
The present paper studies the restriction of the depth index to the set of partitions with blocks of size 2, which can be interpreted as perfect matchings on the set $[2n]$. Closely related issues were studied in the recent paper~\cite{SZ}.

The set of perfect matchings is a classical combinatorial object that was studied in various contexts. Perfect matchings, when viewed as fixed-point free involutions on $[2n]$, have a natural poset structure -- namely, the restriction of the strong Bruhat order on $\symm{2n}$. This poset was studied in detail in~\cite{LexicographicShellabilityBruhatInvolutions}. 
It was shown that the length function  of Bruhat order on $PM_{2n}$ is equal to the length function of the poset that was studied by Deodhar and Srinivasan in~\cite{DeodharSrinivasanStatisticOnInvolutions}. It can be easily checked that the classical strong Bruhat order on $PM_{2n}$ and the order on $PM_{2n}$ induced from the order on set partitions studied in~\cite{StirlingPosets}, are different. 

 In this paper we prove  that generating polynomials of depth-index on $PM_{2n}$ and Bruhat order length on $PM_{2n}$ are essentially equal. More precisely, we prove the following identity:
$$\sum_{\pi\in PM_{2n} }q^{\mathrm{t}(\pi)}=q^{\binom{n+1}{2}}[2n-1]_q!!\,\,\,,$$
where $\mathrm{t}(\pi)$ denotes the depth-index of a set partition $\pi$. This identity implies the main result of the present paper:  
$$I_{PM_{2n} }(q) = \sum_{\pi\in PM_{2n} }q^{\inumber(\pi) }=q^{\binom{n}{2}}[2n-1]_q!!\,\,\,,
$$
where $\inumber(\pi)$ denotes the intertwining number of a set partition $\pi$.

\section{Background}
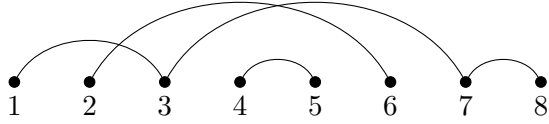
\begin{figure}
\centering
		
\begin{tikzpicture}[scale=1]
\node[below=.05cm] at (0,0) {$1$};
\node[draw,circle, inner sep=0pt, minimum width=4pt, fill=black] (0) at (0,0) {};
\node[below=.05cm] at (1,0) {$2$};
\node[draw,circle, inner sep=0pt, minimum width=4pt, fill=black] (1) at (1,0) {};
\node[below=.05cm] at (2,0) {$3$};
\node[draw,circle, inner sep=0pt, minimum width=4pt, fill=black] (2) at (2,0) {};
\node[below=.05cm] at (3,0) {$4$};
\node[draw,circle, inner sep=0pt, minimum width=4pt, fill=black] (3) at (3,0) {};
\node[below=.05cm] at (4,0) {$5$};
\node[draw,circle, inner sep=0pt, minimum width=4pt, fill=black] (4) at (4,0) {};
\node[below=.05cm] at (5,0) {$6$};
\node[draw,circle, inner sep=0pt, minimum width=4pt, fill=black] (5) at (5,0) {};
\node[below=.05cm] at (6,0) {$7$};
\node[draw,circle, inner sep=0pt, minimum width=4pt, fill=black] (6) at (6,0) {};
\node[below=.05cm] at (7,0) {$8$};
\node[draw,circle, inner sep=0pt, minimum width=4pt, fill=black] (7) at (7,0) {};
\draw[color=black] (2) to [out=120,in=60] (0);
\draw[color=black] (5) to [out=120,in=60] (1);
\draw[color=black] (6) to [out=120,in=60] (2);
\draw[color=black] (7) to [out=120,in=60] (6);
\draw[color=black] (4) to [out=120,in=60] (3);
\end{tikzpicture}
\caption{The arc diagram of the set partition $\pi = 1378|26|45$.}
\label{fig:arc_diagram_example}
\end{figure}

\begin{definition}\label{defn:set_partition}
A collection of subsets $\left\{S_i\right\}_{i\in I}$ of a set $S$ is said to be a {\em set partition} of $S$ if the sets $S_i$ are mutually disjoint and $\bigcup\limits_{i \in I} S_i = S$. The sets $S_i$ are called the blocks of the partition. For $n>0$, the set of all set partitions of $[n]$ is denoted by $\Pi_n$, where $[n]=\{1, 2,\ldots, n\}$. We will often drop set parentheses and commas and just put vertical bars between blocks. If $B_1,\dots,B_k$ are the blocks of the set partition $\pi\in\Pi_n$, then the {\em standard form} of $\pi$ is defined as $B_1|\dots|B_k$ where we assume that $\min B_i < \min B_{i+1}$ for every $1\leq i \leq k-1$, and the elements of each block are listed in increasing order. For example, $\pi = 1378|26|45$ is the standard form of a set partition in $\Pi_8$.
\end{definition}
	
\begin{definition}\label{defn:openers_closers}
Let $\pi\in\Pi_n$ be a set partition. The minimal elements of the blocks are the {\em openers of $\pi$}, and the maximal elements are the {\em closers of $\pi$}.
\end{definition}
	
\begin{example}
The partition $\pi=1378|26|45$ in $\Pi_8$ has openers $1,2,4$ and closers $5,6,8$.
\end{example}
	
\begin{definition}\label{defn:arc_diagram}
Let $\pi\in\Pi_n$ be a set partition with standard form $B_1|\dots|B_k$. The arc diagram of $\pi$ is obtained by placing the labels $1,\dots,n$ on a horizontal line and connecting consecutive elements of each block by an arc, such that the arcs don't intersect if they don't need, as in Figure \ref{fig:arc_diagram_example}.
\end{definition}
	
\begin{definition}\label{defn:extended_arc_diagram}
Let $\pi\in\Pi_n$ be a set partition of $[n]$. The extended arc diagram of $\pi$ is obtained from the arc diagram of $\pi$ by adding a half-arc $(-\infty,i)$ from the far left to each opener $i$ and a half-arc $(j, \infty)$ from each closer $j$ to the far right. These arcs are drawn in such a way that the half-arcs to the left do not cross, and half-arcs to the right do not cross either. See Figure \ref{fig:extended_arc_diagram_example} for an example.
\end{definition}

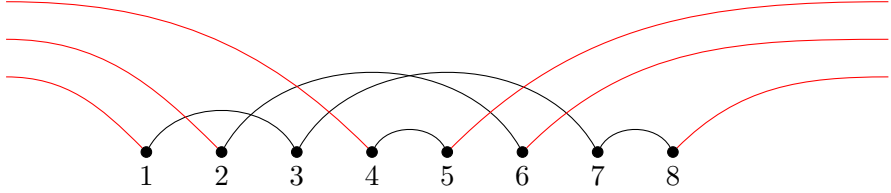
\begin{figure}[h]
\centering
		
\begin{tikzpicture}[scale=1]
\node[below=.05cm] at (0,0) {$1$};
\node[draw,circle, inner sep=0pt, minimum width=4pt, fill=black] (0) at (0,0) {};
\node[below=.05cm] at (1,0) {$2$};
\node[draw,circle, inner sep=0pt, minimum width=4pt, fill=black] (1) at (1,0) {};
\node[below=.05cm] at (2,0) {$3$};
\node[draw,circle, inner sep=0pt, minimum width=4pt, fill=black] (2) at (2,0) {};
\node[below=.05cm] at (3,0) {$4$};
\node[draw,circle, inner sep=0pt, minimum width=4pt, fill=black] (3) at (3,0) {};
\node[below=.05cm] at (4,0) {$5$};
\node[draw,circle, inner sep=0pt, minimum width=4pt, fill=black] (4) at (4,0) {};
\node[below=.05cm] at (5,0) {$6$};
\node[draw,circle, inner sep=0pt, minimum width=4pt, fill=black] (5) at (5,0) {};
\node[below=.05cm] at (6,0) {$7$};
\node[draw,circle, inner sep=0pt, minimum width=4pt, fill=black] (6) at (6,0) {};
\node[below=.05cm] at (7,0) {$8$};
\node[draw,circle, inner sep=0pt, minimum width=4pt, fill=black] (7) at (7,0) {};
\draw[color=black] (2) to [out=120,in=60] (0);
\draw[color=black] (6) to [out=120,in=60] (2);
\draw[color=black] (7) to [out=120,in=60] (6);
\draw[color=black] (5) to [out=120,in=60] (1);
\draw[color=black] (4) to [out=120,in=60] (3);
% left lines
\node[] (a) at (-2,1) {};
\node[] (b) at (-2,1.5) {};
\node[] (c) at (-2,2) {};
\draw[color=red] (0) to [out=135,in=0] (a);
\draw[color=red] (1) to [out=135,in=0] (b);
\draw[color=red] (3) to [out=135,in=0] (c);
% right lines
\node[] (d) at (10,1) {};
\node[] (e) at (10,1.5) {};
\node[] (f) at (10,2) {};
\draw[color=red] (7) to [out=45,in=180] (d);
\draw[color=red] (5) to [out=45,in=180] (e);
\draw[color=red] (4) to [out=45,in=180] (f);

\end{tikzpicture}
\caption{The extended arc diagram of the set partition $\pi = 1378|26|45$.}
\label{fig:extended_arc_diagram_example}
\end{figure}
	
\begin{remark}
The correspondence between set partitions of $[n]$ and their arc diagrams is a bijection, and in the sequel we will identify the partition $\pi$ of $[n]$ with its arc diagram $A(\pi)$. 
\end{remark}

The intertwining number was introduced in \cite{EhrenborgReaddy}. Using the extended diagram representations of set partitions, it can be defined equivalently as in \cite{GeometricInterpretationIntertwiningNumber}.

\begin{definition}\label{defn:intertwining_number}
Let $A\in\Pi_n$ be a set partition of $[n]$. Two (generalized) arcs $(i,j)$ and $(k,l)$ in the extended arc diagram of $A$ are said to cross in $\pi$ if $i < k < j < l$. The number of crossings in the extended arc diagram of $A$ is called the {\em intertwining number} of $A$ and denoted $\inumber(A)$. (See \cite{GeometricInterpretationIntertwiningNumber} for further details and references). 
\end{definition}
	
The depth-index, defined as follows, is closely related to the intertwining number.
	
\begin{definition}\label{definition:depth_index}
Let $A\in\Pi_n$ be a set partition of $[n]$. Denote by $\arcs(A)$ the set of arcs of $A$. For every $1\leq v \leq n$, the {\em depth} of the vertex $v$, denoted by $\depth(v)$, is the number of arcs $(i,j)\in Arcs(A)$ with $1\leq i < v < j \leq n$. For every $\alpha=(u,v) \in \arcs(A)$, the {\em depth of the arc} $\alpha$ is the number of arcs $(i,j)\in\arcs(A)$ with $1\leq i < u < v < j \leq n$.
		
The {\em depth index} $\dindex(A)$ of $A$ is  
$$
\dindex(A)=\sum_{i=1}^{|\arcs(A)|}(n-i)-\sum_{v=1}^n \depth(v) + \sum_{\alpha \in \arcs(A)} \depth(\alpha).
$$ 
\end{definition}

\begin{remark}
Intuitively, the depth of a vertex or an arc is the number of arcs above it in the arc diagram.
\end{remark}
	
Intertwining number and depth-index satisfy the following identity.
\begin{theorem}\cite[Theorem~1]{GeometricInterpretationIntertwiningNumber}\label{thm:dindex_inumber_sum}
For any set partition $A\in\Pi_n$, we have $$\dindex(A)+\inumber(A)=\binom{n}{2}.$$
\end{theorem}
	
We are interested in the distribution of the depth-index on perfect matchings on $[2n]$.
\begin{example}
For the set partition $\pi=1378|26|45$ as in Figure~\ref{fig:extended_arc_diagram_example}, the number of crossings in the extended arc diagram is $\inumber(\pi)=8$. 
The number of arcs in $\pi$ is $5$, the only arc of nonzero depth is $(4,5)$ with $\depth((4,5))=2$, and we have 
\begin{align*}
\depth(2)=\depth(3)=\depth(6) = 1,\\ \depth(4)=\depth(5)=2,\\ \depth(1)=\depth(7)=\depth(8)=0.
\end{align*}
Hence, by definition we have 
\begin{align*}
    \dindex(\pi) 
    &= \sum_{i=1}^8 (8-i)-\sum_{v=1}^{8}\depth(v)+ \sum_{\alpha\in\arcs(\pi)}\depth(\alpha) \\
    &= (7+6+5+4+3) - (1 + 2+ 2+1+ 1) + 2= 20.
\end{align*}
Thus, indeed $\inumber(\pi)+\dindex(\pi)=28=\binom{8}{2}$.
\end{example}

\begin{definition}
A {\em perfect matching} on $[2n]$ is a set partition of $[2n]$ of which every block has size $2$, or equivalently an arc diagram in which every vertex touches exactly one arc. 
The set of perfect matchings on $[2n]$ will be denoted by $PM_{2n}$.
\end{definition}

\begin{remark}
By interpreting the arcs in the arc diagram of a perfect matching $\pi\in PM_{2n}$ as transpositions in $\symm{2n}$, we obtain a {\em fixed point free involution} in $\symm{2n}.$
In fact, this is a  bijection from perfect matchings to fixed point free involutions in $\symm{2n}$. In the rest of the text, every perfect matching $\pi\in PM_{2n}$ will be identified with the corresponding fixed point free involution.
\end{remark}

The length (or rank) functions of the Bruhat-Chevalley order on fixed point free involutions of $\symm{2n}$ and of the order described in \cite{DeodharSrinivasanStatisticOnInvolutions} are equal (see \cite[Corollary~4]{LexicographicShellabilityBruhatInvolutions}). We denote this length function by $\ell$.

We proceed to describe this function in terms of \cite{DeodharSrinivasanStatisticOnInvolutions}.
\begin{definition}
Let $\pi \in PM_{2n}$ be a perfect matching on $[2n]$. The {\em crossing number}  $\cro(\pi)$ of $\pi$ is the number of pairs of arcs $(i,j), (k,l)\in \arcs(\pi)$ that cross.
The {\em span} of an arc $\alpha=(i,j)\in\arcs(\pi)$ 
	is $\aspan(\alpha)=j-i-1$, which is exactly the number of vertices below the arc $\alpha$.
\end{definition}
Let us recall the formula for the generating polynomial of the length function $\ell$ on $PM_{2n}$
\begin{theorem}\cite[Theorem~1.3]{DeodharSrinivasanStatisticOnInvolutions}\label{thm:length_generating_formula}
\begin{enumerate}
	\item	For every $\pi\in PM_{2n}$, 
	$$ 
	\ell(\pi)=\sum_{\alpha\in\arcs(\pi)}\aspan(\alpha)-\cro(\pi)
	$$
	\item The generating polynomial of $\ell$ on $PM_{2n}$ is 
	$$L_{PM_{2n} }(q) 
	= \sum_{\pi\in PM_{2n}}q^{\ell(\pi)} = [2n-1]_q!!,$$
\end{enumerate}

where $[2n-1]_q!!$ is the {\em $q$-analogue of the double factorial}, defined by
    $$[2n-1]_q!!=\prod_{i=1}^n\frac{1-q^{2i-1}}{1-q}.$$
\end{theorem}

\section{Main result}
Our main result is Theorem~\ref{mainitnnn} below which claims that the generating polynomial of the intertwining number $I_{PM_{2n}}(q)$ is
$$I_{PM_{2n} }(q) = \sum_{\pi\in PM_{2n} }q^{\inumber(\pi) }=q^{\binom{n}{2}}[2n-1]_q!!\,\,.
$$
First, we prove the following relation between the distributions of the depth-index and the length function of Bruhat-Chevalley order on fixed point free involutions, and then we establish the connection between the distributions of the intertwining number and the length function of Bruhat-Chevalley order on fixed point free involutions which is our main result.
\begin{theorem}\label{thm:main}
	The depth index on $PM_{2n}$ is distributed as follows:
	$$T_{PM_{2n}}(q)=\sum_{\pi\in PM_{2n} }q^{\dindex(\pi) }  =q^{\binom{n+1}{2} }[2n-1]_q!!.
	$$
\end{theorem}

\begin{corollary}
The depth index $\dindex$ and $\binom{n+1}{2}+\ell$ are equidistributed on $PM_{2n}$.
\end{corollary}

The following definitions are valid for arcs in the regular arc diagrams and not the extended ones.  
\begin{definition}\cite{CrossingNestingDistribution}
	Let $\pi\in PM_{2n}$. We say that two arcs $e=(i,j),f=(k,l)\in\arcs(\pi)$ form:
	\begin{itemize}
		\item[(i)] a {\em crossing} with $e$ as {\em initial edge} if $i < k < j < l$;
		\item[(ii)] a {\em nesting} with $e$ as {\em initial edge} if $i < k < l < j$;
		\item[(iii)] a {\em alignment} with $e$ as {\em initial edge} if $i < j < k < l$.
	\end{itemize}
The number of crossings, nestings and alignments will be denoted by $\cro,\nst$ and $\al$ respectively. 
\end{definition}

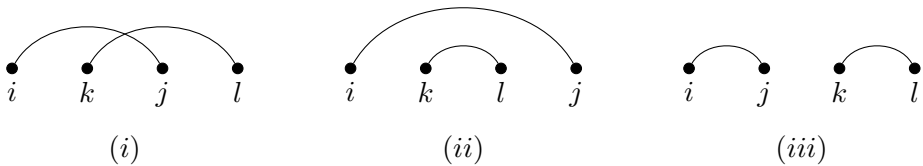
\begin{figure}[h]
  \centering
% nodes
\begin{tikzpicture}[scale=1]
\node[below=.05cm] at (0,0) {$i$};
\node[draw,circle, inner sep=0pt, minimum width=4pt, fill=black] (0) at (0,0) {};
\node[below=.05cm] at (1,0) {$k$};
\node[draw,circle, inner sep=0pt, minimum width=4pt, fill=black] (1) at (1,0) {};
\node[below=.05cm] at (2,0) {$j$};
\node[draw,circle, inner sep=0pt, minimum width=4pt, fill=black] (2) at (2,0) {};
\node[below=.05cm] at (3,0) {$l$};
\node[draw,circle, inner sep=0pt, minimum width=4pt, fill=black] (3) at (3,0) {};
\node[below=.05cm] at (4.5,0) {$i$};
\node[draw,circle, inner sep=0pt, minimum width=4pt, fill=black] (4) at (4.5,0) {};
\node[below=.05cm] at (5.5,0) {$k$};
\node[draw,circle, inner sep=0pt, minimum width=4pt, fill=black] (5) at (5.5,0) {};
\node[below=.05cm] at (6.5,0) {$l$};
\node[draw,circle, inner sep=0pt, minimum width=4pt, fill=black] (6) at (6.5,0) {};
\node[below=.05cm] at (7.5,0) {$j$};
\node[draw,circle, inner sep=0pt, minimum width=4pt, fill=black] (7) at (7.5,0) {};
\node[below=.05cm] at (9,0) {$i$};
\node[draw,circle, inner sep=0pt, minimum width=4pt, fill=black] (8) at (9,0) {};
\node[below=.05cm] at (10,0) {$j$};
\node[draw,circle, inner sep=0pt, minimum width=4pt, fill=black] (9) at (10,0) {};
\node[below=.05cm] at (11,0) {$k$};
\node[draw,circle, inner sep=0pt, minimum width=4pt, fill=black] (10) at (11,0) {};
\node[below=.05cm] at (12,0) {$l$};
\node[draw,circle, inner sep=0pt, minimum width=4pt, fill=black] (11) at (12,0) {};
% arcs
\draw[color=black] (2) to [out=120,in=60] (0);
\draw[color=black] (3) to [out=120,in=60] (1);
\draw[color=black] (7) to [out=120,in=60] (4);
\draw[color=black] (6) to [out=120,in=60] (5);
\draw[color=black] (9) to [out=120,in=60] (8);
\draw[color=black] (11) to [out=120,in=60] (10);

%labels of arc relations
\node[below=0.75cm] at (1.5,0) {$(i)$};
\node[below=0.75cm] at (6,0) {$(ii)$};
\node[below=0.75cm] at (10.5,0) {${(iii)}$};

\end{tikzpicture}
\caption{Crossing, nesting and alignment of two edges}
  \label{fig:crossing_nestings_alignements}
\end{figure}

\begin{lemma}\label{obs:nst_cr_al_sum}
	For every perfect matching $\pi$ in $PM_{2n}$  we have $$\cro(\pi)+\nst(\pi)+\al(\pi) = \binom{n}{2}\,\,.$$
\end{lemma}
\begin{proof}
For every perfect matching $\pi$ in $PM_{2n}$ there are exactly $\binom{n}{2}$ pairs $(e,f) \in \arcs(\pi) \times \arcs(\pi)$ with $e=(i,j)$, $f=(k,l)$ and $i < k$. For every such pair $(e,f)$ exactly one of the three following possibilities holds: $(e,f)$ forms a crossing' or a nesting' or an alignment. Thus, the sum of the numbers of all crossings, nestings and alignments of a given perfect matching $\pi\in PM_{2n}$ equals to the number of all pairs $(e,f) \in \arcs(\pi) \times \arcs(\pi)$ with $e=(i,j)$, $f=(k,l)$ and $i < k$, i.e., it equals to $\binom{n}{2}$. 
\end{proof}

By changing the order of summation, we obtain the following identities:
\begin{lemma}\label{obs:depth_identity} For every $\pi\in PM_{2n}$, we have:
	\begin{enumerate}
	\item $\sum_{v=1}^{2n} \depth(v) = \sum_{\alpha \in \arcs(\pi)}\aspan(\alpha)$;
	\item $\nst(\pi) = \sum_{\alpha\in\arcs(\pi)} \depth(\alpha)$.
	\end{enumerate}
\end{lemma}
\begin{proof}
Recall that $\aspan(\alpha)$ is the number of vertices below the arc $\alpha$. Thus, each vertex contributes the number of arcs which are above it to the r.h.s. of the first identity. Recall that by definition the depth of a vertex is the number of arcs above it. Thus, the sum $\sum_{\alpha \in \arcs(\pi)}\aspan(\alpha)$ counts the sum of depths of all the vertices which is exactly the l.h.s. of the first identity.

Recall that by definition the depth of an arc is the number of arcs above it. Thus, each pair of arcs which forms a nesting contributes exactly one to the sum $\sum_{\alpha\in\arcs(\pi)} \depth(\alpha)$. The number of all pairs of arcs which form nestings is, by definition, $\nst(\pi)$. This proves the second identity.  
\end{proof}

\begin{definition}\label{definition:tvd}
	Let $\pi \in PM_{2n}$. The sum $\sum_{v=1}^{2n}\depth(v)=\sum_{\alpha\in\arcs(\pi)}\aspan(\alpha)$ is called the {\em total vertex depth} of $\pi$ and denoted by $\tvd(\pi)$.
\end{definition}

By combining Lemma~\ref{obs:nst_cr_al_sum} and Lemma~\ref{obs:depth_identity} we obtain the following expressions for the depth index.
\begin{lemma}\label{obs:depth-index-restatement}
	For every $\pi\in PM_{2n}$ we have the following equality:
	\begin{align*}
	\dindex(\pi) 
		&= \sum_{i=1}^{|\arcs(\pi)|}(2n-i) - \tvd(\pi) + \nst(\pi) \\
		&= n^2 + 2\binom{n}{2} - \tvd(\pi) - \cro(\pi) - \al(\pi).		  
	\end{align*}
\end{lemma}
\begin{proof} Since $|\arcs(\pi)|=n$, $\tvd(\pi)=\sum_{v=1}^{2n} \depth(v)$, $\sum_{\alpha \in \arcs(\pi)} \depth(\alpha)=\nst(\pi)$ and $\nst(\pi)=\binom{n}{2}-\cro(\pi)-\al(\pi)$, we have
\begin{align*}  
\dindex(\pi)&=\sum_{i=1}^{|\arcs(\pi)|}(2n-i)-\sum_{v=1}^{2n} \depth(v) + \sum_{\alpha \in \arcs(\pi)} \depth(\alpha) \\
&=\sum_{i=1}^{n}(2n-i)-\sum_{v=1}^{2n} \depth(v) + \sum_{\alpha \in \arcs(\pi)} \depth(\alpha) \\
&=\binom{2n}{2}-\binom{n}{2}-\tvd(\pi)+\nst(\pi) \\
&=\binom{2n}{2}-\binom{n}{2}-\tvd(\pi)+\binom{n}{2}-\cro(\pi)-\al(\pi) \\
&=\binom{2n}{2}-\tvd(\pi)-\cro(\pi)-\al(\pi) \\
&=n^2+n^2-n-\tvd(\pi)-\cro(\pi)-\al(\pi) \\
&= n^2 + 2\binom{n}{2} - \tvd(\pi) - \cro(\pi) - \al(\pi).
\end{align*}
\end{proof}

Recall the definition of the intertwining number (Definition \ref{defn:intertwining_number}): the number of crossings in the extended arc diagram of $A$ is called the {\em intertwining number} of $A$ and is denoted $\inumber(A)$.
\begin{lemma}\label{lemma:intertwining_number_restatement}
Let $\pi\in PM_{2n}$. Then $\inumber(\pi) = 3\cro(\pi) + 2\nst(\pi) + \al(\pi)$.
\end{lemma}
\begin{proof}
	Consider two arcs of $\pi$: $e=(i,j)$ and $f=(k,l)$ with $i < k$.
	If the pair $(e,f)$ is a nesting,then it gives exactly two crossings in the extended arc diagram: the half-arc $(-\infty,k)$ crosses the arc $e=(i, j)$, and the arc $e$ crosses the half-arc $(l,\infty)$.
	If the pair $(e,f)$ is a crossing, it contributes exactly three crossings to the extended arc diagram: the arc $e$ crosses the arc $f$, the half-arc $(-\infty,k)$ crosses the arc $e$, and the arc $f$ crosses the half-arc $(j,\infty)$. Finally, if the pair $(e,f)$ is an alignment, we have a single crossing between the half-arcs $(j,\infty)$ and $(-\infty, k)$.
	By summing everything, we get the desired result.
\end{proof}

By using Theorem \ref{thm:dindex_inumber_sum} one can express $\tvd(\pi)$ in terms of $\nst(\pi), \al(\pi)$ and $\cro(\pi)$.

\begin{lemma}
	For every $\pi\in PM_{2n}$ we have:
	$$\tvd(\pi) = 2\binom{n}{2} - 2\al(\pi) = 2(\cro(\pi) + \nst(\pi)).$$
\end{lemma}

\begin{proof}
	By Theorem~\ref{thm:dindex_inumber_sum}, for $\pi\in PM_{2n}$,
	and using Lemma~\ref{obs:depth-index-restatement} and Lemma~\ref{lemma:intertwining_number_restatement} we obtain
	\begin{align*}
		\binom{2n}{2} 
		&= \dindex(\pi)+\inumber(\pi) \\
		&= n^2 + 2\binom{n}{2} -\tvd(\pi) - \cro(\pi) - \al(\pi) + 3\cro(\pi) + 2\nst(\pi) + \al(\pi) \\
		&= n^2 + 2\binom{n}{2} - \tvd(\pi) + 2\cro(\pi) + 2\nst(\pi).
	\end{align*}
	By rearranging the equation we obtain:
	\begin{align*}
	    \tvd(\pi) 
	    &= n^2 + 2\binom{n}{2 }- \binom{2n}{2} + 2\cro(\pi) + 2\nst(\pi) \\
	    &= 2\cro(\pi) + 2\nst(\pi) \\
	    &= 2\binom{n}{2} - 2\al(\pi),
	\end{align*}
	as desired.
\end{proof}

We can now express $\ell(\pi)$ and $\dindex(\pi)$ using $\nst(\pi)$ and $\cro(\pi)$.
\begin{corollary}\label{cor:statistic_formulas}
For every $\pi\in PM_{2n}$ we have:
$$\ell(\pi)=\cro(\pi)+2\nst(\pi)$$
and 
$$\dindex(\pi)=n^2 + \binom{n}{2}-2\cro(\pi)-\nst(\pi)\,.$$
\end{corollary} 
\begin{proof} By Lemma~\ref{obs:depth_identity}, Definition \ref{definition:tvd} and Theorem~\ref{thm:length_generating_formula}, we have $\ell(\pi) = \tvd(\pi)-\cro(\pi)$. By Definitions \ref{definition:depth_index} and  \ref{definition:tvd} we have $\dindex(\pi)=n^2 + \binom{n}{2} - \tvd(\pi) + \nst(\pi)$.
Replace $\tvd(\pi)$ by $2\nst(\pi)+2\cro(\pi)$ and the result follows.
\end{proof}

We are ready to prove the main theorem.
In order to do it, we rely on the following facts.
\begin{theorem}\cite[Theorem~1.2]{CrossingNestingDistribution}\label{thm:involution_existence}
There exists an involution $\phi:PM_{2n}\to PM_{2n}$ that preserves the number of alignments and exchanges the number of crossings and nestings. In other words, for each $\pi\in PM_{2n}$ we have 
$$\al(\phi(\pi)) = \al(\pi), \cro(\phi(\pi)) = \nst(\pi), \nst(\phi(\pi)) = \cro(\pi).$$
\end{theorem}

\begin{observation}\label{fact:double_factorial_palindrom}
The polynomial $[2n-1]_q!!$ has degree $n^2-n$ and is \em{palindromic}. Namely, the coefficients of $q^r$ and $q^{n^2 - n - r}$ in $[2n-1]_q!!$ are equal for each $0\leq r \leq n^2 - n$.
\end{observation}

\begin{corollary}\label{cor:bruhat_order_inversion}
There exists a bijection  $\psi:PM_{2n}\to PM_{2n}$ such that 
$$\ell(\psi(\pi))=n^2-n-\ell(\pi)$$ 
for every $\pi\in PM_{2n}$.
\end{corollary}
\begin{proof}
By Theorem~\ref{thm:length_generating_formula}, the generating polynomial of the length function on $PM_{2n}$ is $[2n-1]_q!!$. Together with Observation~\ref{fact:double_factorial_palindrom}, it implies that for every $0 \leq r \leq n^2 - n$, the number of elements of $PM_{2n}$ with length $r$ is equal to the number of elements with length $n^2 - n - r$, as desired.
\end{proof}

We proceed to prove the Theorem~\ref{thm:main}.
\begin{proof}[Proof of Theorem~\ref{thm:main}]
By Corollary~\ref{cor:statistic_formulas} and Theorem~\ref{thm:involution_existence} we have:
$$\ell(\phi(\pi)) = 2\nst(\phi(\pi))+\cro(\phi(\pi))=2\cro(\pi)+\nst(\pi).$$
By applying Corollary~\ref{cor:statistic_formulas} again we obtain: 
$$\dindex(\pi) + \ell(\phi(\pi))=n^2+\binom{n}{2}.$$
Using $\psi$ from Corollary~\ref{cor:bruhat_order_inversion} and the relation
$\ell(\phi(\pi))=n^2 - n - \ell(\pi)$ we obtain:
$$\dindex(\pi) + n^2 - n - \ell(\psi(\phi(\pi))) = n^2 + \binom{n}{2}.$$
Finally, by rearranging the equation, we obtain 
$t(\pi)=\binom{n+1}{2}+\ell(\psi(\phi(\pi)))$. Since, $\psi\circ\phi$ is a bijection, the statistics $\dindex$ and $\binom{n+1}{2}+\ell$ are equi-distributed, and the Theorem~\ref{thm:main} follows.
\end{proof}

By using the identity $\inumber(\pi)+\dindex(\pi)=\binom{2n}{2}$ for every $\pi\in\Pi_{2n}$ we obtain our main result, i.e., a formula for the generating polynomial of the intertwining number.
We rely on the following observation.
\begin{observation}\label{fact:palindromic_polynomial_fraq_q}
    If $p(q)$ is a palindromic polynomial in $q$ of degree $m$, then
    $$q^m p\left(\frac{1}{q}\right)=p(q)\,\,.$$
\end{observation}

\begin{theorem}\label{mainitnnn}
The generating polynomial of the intertwining number $I_{PM_{2n}}(q)$ is
$$I_{PM_{2n} }(q) = \sum_{\pi\in PM_{2n} }q^{\inumber(\pi) }=q^{\binom{n}{2}}[2n-1]_q!!\,\,.
$$
\end{theorem}

\begin{proof}
By Theorem~\ref{thm:dindex_inumber_sum} we have
$$\sum_{\pi\in PM_{2n} }q^{\inumber(\pi)}
=\sum_{\pi\in PM_{2n} } q^{\binom{2n}{2} - \dindex(\pi)}
=q^{\binom{2n}{2}} \sum_{\pi\in PM_{2n} } \frac{1}{q^{\dindex(\pi)} },
$$
By applying Theorem~\ref{thm:main}, we obtain
$$q^{\binom{2n}{2}} \sum_{\pi\in PM_{2n} } \frac{1}{q^{\dindex(\pi)} } =
q^{\binom{2n}{2}}\frac{1}{q^{\binom{n+1}{2} } }[2n-1]_{\frac{1}{q} }!!.
$$

By Observation~\ref{fact:palindromic_polynomial_fraq_q} and Observation~\ref{fact:double_factorial_palindrom}, 
$q^{n^2 - n}[2n-1]_{\frac{1}{q} }!!=[2n-1]_q!!$, and therefore 
\begin{align*}
    q^{\binom{2n}{2}}\frac{1}{q^{\binom{n+1}{2} } }[2n-1]_{\frac{1}{q} }!! &= 
    q^{\binom{2n}{2}}\frac{1}{q^{\binom{n+1}{2} }q^{n^2-n} }[2n-1]_q!! \\ &=
    q^{\binom{2n}{2} - \binom{n+1}{2} - n^2 + n}[2n-1]_q!!=q^{\binom{n}{2} }[2n-1]_q!!,
\end{align*}
as desired.
\end{proof}

{\bf Acknowledgements.} We are grateful to Yuval Roichman and Mahir Bilen Can for helpful discussions. We express a special gratitude to Ron Adin for many very valuable remarks. 

{\bf Conflict of interest statement.} On behalf of all authors, the corresponding author states that there is no conflict of interest.

{\bf Contribution of the co-authors.} Both co-authors contributed to all parts of the paper. 
%{\bf Data availability statement} The authors declare that the data supporting the findings of this study are available within the paper.

\end{document}